\documentclass{amsart}

\usepackage{amsthm}
\usepackage{amssymb}
\usepackage{amsmath}
\usepackage{amsmath,amscd}
\usepackage{epsfig}
\usepackage{mathrsfs}
\usepackage{verbatim}

\newtheorem{theorem}{Theorem}
\newtheorem{corollary}{Corollary}
\newtheorem{lemma}{Lemma}
\newtheorem{proposition}{Proposition}

\newtheorem{definition}{Definition}
\newtheorem{remark}{Remark}

\newcommand{\e}[0]{\mathbf e}
\newcommand{\mb}{\mathbb}
\newcommand{\mc}{\mathcal}
\newcommand{\wt}{\widetilde}
\newcommand{\ol}{\overline}

\title{A Decay Estimate For The Stability Operator Of The Helicoid}

\author{Stephen J. Kleene}
\address{Department of Mathematics, Brown University, Providence, RI 02906}

\begin{document}
\maketitle
\begin{abstract}
We consider the Poisson Problem for the stability operator of the helicoid on a  vertical strip, under the assumption that the source term is supported on a strip of fixed height. We prove that solutions decay at a definite rate away from the support assuming natural orthogonality conditions on the source term.
\end{abstract}

\section{introduction}
In this article we consider the Poisson problem
\begin{align} \label{PoissonProblem}
\wt{\mc{L}}u = E
\end{align}
for the operator $\wt{\mc{L}} = \Delta_{\mb{R}^2} + 2 \cosh^{-2}(s)$ on domains of the form $R = [-\ell, \ell] \times \mb{R}$ in $\mb{R}^2$,  where we take $(s, z)$ as coordinates on $\mb{R}^2$ and where the source term  $E$ is supported on the strip $ [-\ell, \ell] \times [-2 \pi, 2 \pi]$. The operator $\wt{\mc{L}}$ is, up to the conformal factor $\cosh^{2}(s)$,  the stability operator of the helicoid in its standard conformal parametrization $F: \mb{R}^2 \rightarrow \mb{R}^3$  given by
\begin{align}\label{HelicoidParametrization}
F(s, z) = \sinh(s) \e_r(z) + z \e_z, \quad \e_{r}(z) = \cos(z) \e_x + \sin(z) \e_y,
\end{align}
and problem (\ref{PoissonProblem}) arises naturally in many geometric problems  involving both the  helicoid and the catenoid. In applications (such has \cite{BK}, \cite{BK2}), one frequently has symmetry assumptions which reduce the domain of (\ref{PoissonProblem}) to long cylinders  $[-\ell, \ell] \times \mb{S}^1$, and where the invertibility properties of the operator $\mc{L}$ are well understood. However,  in many interesting problems one wishes to relax these assumptions, in which case one hopes  symmetry  can be replaced by decay estimates. This article records such a decay estimate for source terms $E$  satisfying certain strong orthogonality conditions arising from  the lower  spectrum of the operator $\wt{\mc{L}}$. Precisely, we assume that the source term $E$ satisifies
\begin{align}\label{StrongOrthogonality}
\int_{-\ell}^{\ell} E(s, z) \cosh^{-1}(s)ds = \int_{-\ell}^{\ell} E(s, z) \tanh(s)ds, \quad z \in \mb{R}.
\end{align}
Our main theorem is then:
 \begin{theorem} \label{RBCSolutionsWeightedControl}
For any $\xi \in (0, 1/2)$, there is a constant $C = C(\xi)$ such that: Given a  $C^{0, \alpha}$ function $E(s, z): [-\ell, \ell] \times \mb{R} \rightarrow \mb{R}$ supported on the strip $[- \ell, \ell] \times [-2 \pi, 2 \pi]$ with $\| E\|_{0, \alpha} \leq \beta$ and satisfying the strong orthogonality conditions (\ref{StrongOrthogonality}), there is a unique locally $C^{2, \alpha}$ function $u: [- \ell, \ell] \times \mb{R}\rightarrow \mb{R}$ solving  (\ref{PoissonProblem}) and satisfying the orthogonality conditions (\ref{StrongOrthogonality}) and such that: 
\begin{enumerate}
\item \label{ODDRBS} The odd part $u^{-}(s ,z) : = \frac{1}{2}\left(u(s, z) - u(-s, z) \right)$ of $u$ satisfies the boundary condition
\[
u^{-}_{, s}(\ell, z) \tanh(\ell) - u^{-}(\ell, z) \tanh'(\ell) = 0, \quad \forall z \in \mb{R}.
\]
\item \label{EVENRBS} The even part $u^{+}(s ,z) : = \frac{1}{2}\left(u(s, z) +u(-s, z) \right)$ of $u$ satisfies the boundary condition
\[
u^{+}_{, s}(\ell, z) \cosh^{-1}(\ell) - u^{+}(\ell, z) \left(\cosh^{-1}\right)'(\ell) = 0, \quad \forall z \in \mb{R}.
\]
\end{enumerate}
Moreover,  $u$ satisfies the weighted estimate
 \begin{align} \notag
\left\| u \right\|_{2, \alpha}  \leq C \ell^{5/2}  \beta \left(\frac{1}{1 + |z|}\right)^{\xi}.
\end{align}

 \end{theorem}
Above, $\|u\|_{k, \alpha}$ denotes the localized $C^{k, \alpha}$ holder norm for a function $u: \Lambda \rightarrow \mb{R}$, so
\[
\| u\|_{k, \alpha} (p) : = \| u: C^{k, \alpha} (\Lambda \cap D_p) \|
\]
where $D_p \subset \mb{R}^2$ denotes the unit disk in $\mb{R}^2$ centered at $p$.

\subsection{Structure of the proof of Theorem \ref{RBCSolutionsIntegralControl}}
The basic strategy for proving Theorem \ref{RBCSolutionsIntegralControl} is to search for solutions $u_N$ on domains $[-\ell, \ell] \times [-N, N]$ with appropriate boundary values, and to prove uniform integral  estimates for $u_N$ in $N$,  which can then be improved to decay estimates. For fixed large $N$ and $\ell$ it can be shown that the operator $\wt{\mc{L}}$ has trivial dirichlet kernel on  $[-\ell, \ell] \times [-N, N]$, however the gap at $0$ decays exponentially in $\ell$. Thus, though solutions to (\ref{PoissonProblem}) with Dirichlet boundary conditions on $[-\ell, \ell] \times [-N, N]$ exist without assuming any orthogonality on the source term $E$, their bounds are poorly behaved in $\ell$. The bounded kernel of $\wt{\mc{L}}$ on $\mb{R}^2$ is spanned by the functions $\tanh(s)$, $\cos(z) \cosh^{-1}(s)$, and $\sin(z) \cosh^{-1}(s)$, and these functions are the principal obstructions to obtaining uniform integral control of the solutions.  The Poisson problem  (\ref{PoissonProblem}) can be solved  with boundary conditions (\ref{ODDRBS}) and (\ref{EVENRBS}) in the Statement of Theorem \ref{RBCSolutionsWeightedControl} on the odd and even parts of $u_{N}$, respectively--as well as Neumann condition along the boundary component $z = \pm N$--and the strong orthogonality condition (\ref{StrongOrthogonality}) on  the source term is  inherited by the solutions $u_N$. We can then apply a one-dimensional weighted  Poincare Inequality to obtain $L^2$  and this $C^{2, \alpha}$ control on the solutions $u_N$, which thus converge on compact subsets of the strip $[-\ell, \ell] \times \mb{R}$ to a limiting function $u$ solving (\ref{PoissonProblem}). The Poincare Inequality then implies convexity of the integrals $\int_{-\ell}^\ell u^2 ds$ which, combined with the uniform integral control implies decay. Schauder estimates and a bootstrapping argument then improves the integral decay to pointwise decay. We Remark that the arguments presented in this paper can likely be generalized to operators other than stability of the Helicoid,  however we have not pursued this here.

\subsection{Outline of the article}
In Section \ref{Sec:AprioriEstimates} we prove Proposition \ref{RBCSolutionsIntegralControl},  which records a-priori  $L^2$ and $C^{0}$  estimates  for solutions to (\ref{PoissonProblem}). We impose boundary conditions on the even and odd parts separately which allow for the orthogonality conditions  (\ref{StrongOrthogonality}) imposed on source terms to be recovered in the solutions. As a consequence, we characterize the kernel of the operator $\wt{\mc{L}}$ with these imposed boundary values in Corollaries \ref{WhatsInTheRobinKernel} and \ref{WhatsInTheRobinKernelAgain}. The principal tool used in the proof of   Proposition \ref{RBCSolutionsIntegralControl} is the one-dimensional weighted Poincare Inequality recorded in Propositions \ref{PoincareInequality} and \ref{PoincareInequalityEven}.  Section \ref{Sec:Solvability}   records Proposition \ref{StripInvertibilityWithRBC}, which proves the existence of solutions  $u_N$ to (\ref{PoissonProblem}) for  source terms satisfying the strong orthogonality conditions  with  uniform $L^2$ and $C^{0}$ estimates on arbitrarily tall rectangles $[-\ell, \ell] \times [-N, N]$. Proposition   \ref{RBCSolutionsWeightedControl} then  follows from a bootstrapping argument using  local estimates (Schauder Theory)  and the uniform $L^2$ control on the solutions. This is recorded in Section \ref{Sec:ProofOfMainTheorem}.  
  
  \section{Apriori estimates for solutions:  Proposition \ref{RBCSolutionsIntegralControl} and one dimension Poincare inequalities}\label{Sec:AprioriEstimates}
 \begin{proposition} \label{RBCSolutionsIntegralControl}
 There is a universal constant $C$ so that: Given a function  $u(s, z): [-\ell, \ell] \times [- N, N] \rightarrow \mb{R}$ satisfying the following conditions:
\begin{enumerate}
\item \label{SolutionRobinBoundaryCondition} The odd part $u^{-}$ of $u$ satisfies the Robin boundary condition
\begin{align} \notag
\partial_s u^{-} (\ell, z) \tanh(\ell) - u^{-}_N(\ell, z)/\cosh^2(\ell) = 0.
\end{align}

\item \label{SolutionRobinBoundaryCondition2}  The even part  $u^{+}$ of $u$ satisfies the Robin boundary condition
\begin{align} \notag
\partial_su^{+} (\ell, z) \cosh^{-1}(\ell) - u^{+}(\ell, z)\left(\cosh^{-1}\right)'(\ell) = 0.
\end{align}

\item \label{SolutionNeumanBoundaryCondition} $u$ satisfies the Neuman boundary condition
\begin{align} \notag
\partial_z u(s, \pm N) = 0.
\end{align}
\item  \label{SolutionStrongOrthogonality}  $u$ satisfies the orthogonality conditions
\[
\int_{-\ell}^{\ell} u (s, z)\tanh(s) ds = \int_{-\ell}^{\ell} u (s, z)\cosh^{-1}(s) ds, \quad \forall z
\]
\end{enumerate}
Then it holds that 
 \begin{align} \notag
\sup |u_N|, \quad  \| u_N \|_{L^2}   \leq C \ell^2 \left(\| \wt{\mc{L}} u\|_{L^2}\right)
 \end{align}
 \end{proposition}
 \begin{definition} \label{Dim1Energy}
For an odd function  $f$ belonging to the space $W^{1, 2} ([-\ell, \ell])$ we set
 \begin{align} \notag
 e^{-}_\ell (f) : = \int_0^\ell f'^2 (s) ds - 2 \int_0^\ell f^2(s) \cosh^{-2} (s) ds - f^2 (\ell) \tanh'(\ell)/\tanh(\ell).
 \end{align}
 \end{definition}

\begin{proposition}[One dimensional weighted Poincare Inequality, odd case] \label{PoincareInequality}
There is a universal constant  $P > 0$ independent of $\ell$ so that: Let $f \in W^{1, 2}([-\ell, \ell])$  be an odd function satisfying the orthogonality condition:
\begin{align} \label{WeightedOrthogonalityCondition}
\int_0^\ell f (s)\tanh(s)/\cosh^2(s) ds = 0.
\end{align}
Then it holds that 
\begin{align} \notag
e^{-}_\ell(f) & \geq P \int_0^\ell f^2(s) \cosh^{-2} (s)ds
\end{align}
\end{proposition}

Before  proving Proposition \ref{PoincareInequality}, we characterize  the bottom spectrum of the operator $\wt{L} = \partial_{ss} + 2 \cosh^{-2}(s)$.

\begin{lemma}[Lower weighted eigenspace for $\wt{L}$]\label{LEigenvalues}
Let $\phi: \mb{R} \rightarrow \mb{R}$ be a solution to the weighted eigenvalue problem for $\wt{L}$ below:
\[
\wt{L} \phi = \lambda \cosh^{-2}(s)\phi.
\]
Assume that 
\[
\int_{-\infty}^\infty \phi^2 \cosh^{-2}(s) ds + \int_{-\infty}^\infty \left(\phi'(s)\right)^2 ds < \infty
\]
and   $\lambda \leq 0$. Then $\lambda = -2$, in which case $\phi$ is  constant, or else $\lambda = 0$ and $\phi$ is a multiple of $\tanh(s)$. 
 \end{lemma}
 
 \begin{proof}
 The Gauss map $N: \mb{R}^2 \rightarrow \mb{S}^2$ of the immersion $F$ defined in (\ref{HelicoidParametrization}) is (anti) conformal with conformal factor $\cosh^{-2}(s)$. Thus solutions to the weighted eigenvalue problem in the statement of the lemma correspond to rotationally invariant (about the $z$-axis) $W^{1,2 }$ solutions to 
 \[
 \Delta_{\mb{S}^2} \phi + 2 \phi = \lambda \phi
 \]
 on the sphere.  It is well-known that the eigenvalues  of the laplacian on $\mb{S}^2$ is $\lambda_{k} = k(k + 1)$ with multiplicity $2k + 1$. The eigenspace corresponding to $\lambda_0 = 0$ is spanned by constants, while the eigenspace $\lambda_1 = 2$ is spanned by the coordinate functions, $x, y, z$. The functions in these spaces that survive the rotational symmetry condition are $z$ and $1$. Since the pullback $N^* z$  of $z$ under $N$ is $\tanh(s)$, this completes the proof. 
  \end{proof}

\begin{proof}[Proof of Proposition \ref{PoincareInequality}]
Observe first that  $e^{-}_\ell(f)$ is bounded below for each $\ell$ on $L^2$ normalized functions: We have
\[
f^2(\ell) = \left( \int_{0}^\ell f'(s) ds\right)^2 \leq \ell \int_{0}^\ell f'^2(s) ds
\]
and thus, 
\begin{align}\label{ASDF34fr}
e^{-}_\ell(f) & = \int_0^\ell f'^2 (s) ds - 2 \int_0^\ell f^2(s) \cosh^{-2} (s) ds - f^2 (\ell) \tanh'(\ell)/\tanh(\ell) \\ \notag
& \geq \int_0^\ell f'^2 (s) ds - 2 \int_0^\ell f^2(s) \cosh^{-2} (s) ds - \ell \tanh'(\ell)/\tanh(\ell) \int_0^\ell f'^2 (s) ds \\ \notag
& = \left (1 -  \ell \tanh'(\ell)/\tanh(\ell) \right) \int_{0}^\ell f'^2(s) ds- 2 \int_0^\ell f^2(s) \cosh^{-2} (s) ds  \\ \notag
& \geq - 2 \int_0^\ell f^2(s) \cosh^{-2} (s) ds .
\end{align}
Now, for  each fixed $\ell$, we consider a minimizing sequence $f_j$ for $e^{-}_\ell$ with
\[
\int_0^\ell f^2_j \cosh^{-2} (s) ds = 1.
\]
 By  (\ref{ASDF34fr}), the sequence is uniformly bounded in $W^{1, 2}\left([0, \ell]\right)$, and thus, after passing to a subsequence, we can assume that $f_j$  strongly  converges in $L^2$ and weakly in $W^{1, 2}$,  to a limiting function $\phi_{\ell}$ in $W^{1, 2}([0, \ell])$. The function $\phi_\ell$ is then  a smooth solution to
 \begin{align}\label{eigen}
 \wt{L} \phi_\ell= \lambda_{\ell} \cosh^{-2}(s)\phi_{\ell}, \quad \lambda_{\ell} = \inf_{f \in W^{1, 2}[0, \ell]} e^{-}_\ell(f)
 \end{align} 
and  satisfies the orthogonality condition 
\begin{align}\label{orthco}
\int_{0}^\ell \phi_{\ell}(s) \tanh(s)/\cosh^{2}(s) ds = 0.
\end{align}
 We now consider a sequence $\ell_k \rightarrow \infty$ and  set 
 \[
 \lambda : = \liminf_{k \rightarrow \infty} \lambda_{k}
 \]
 where $\lambda_{k} = \lambda_{\ell_k}$. Let $\phi_k$ denote the corresponding lowest eigenfunction $\phi_k$ satisfying  (\ref{eigen}) with $\ell  = \ell_k$, and assume that $\phi_k$ is normalized so that 
 \[
 \int_{0}^{\ell_k} \phi^2_{k}\cosh^{-2}(s) ds = 1. 
 \]
   The uniform energy bound on the sequence $\{ \phi_k\}$ gives that 
\begin{align} \label{enbd}
|\phi_k (s)| \leq C s^{1/2}
\end{align}
and thus the sequence $\phi_k$ converges on compact subsets of $[0, \infty]$ to a smooth limit $f$ satisfying
\[
\wt{\mc{L}} f = \lambda \cosh^{-2}(s) f.
\]
as well as the orthogonality condition (\ref{orthco}) with $\ell = \infty$. Moreover, the uniform  bound (\ref{enbd}) and the dominated convergence therorem gives
\[
\int_{0}^\ell f^{2}(s) \cosh^{-2} (s) ds = 1.
\]
By Lemma \ref{LEigenvalues},  $\lambda > 0$ since $f$ is odd and satisfies  (\ref{orthco}). This completes the proof.
\end{proof}
We now record similar results in the case of even functions. 
\begin{definition} \label{Dim1EnergyEven}
For an even function  $f$ belonging to the space $W^{1, 2} ([-\ell, \ell])$ we set
 \begin{align} \notag
 e^{+}_\ell (f) & : = \int_{0}^\ell f'^2 (s) ds - 2 \int_0^\ell f^2(s) \cosh^{-2} (s) ds - f^2 (\ell) (\cosh^{-1})'(\ell)/\cosh^{-1}(\ell) \\
 & =  \int_{0}^\ell f'^2 (s) ds - 2 \int_0^\ell f^2(s) \cosh^{-2} (s) ds + f^2 (\ell) \tanh(\ell).
 \end{align}
 \end{definition}

\begin{proposition}[One dimensional weighted Poincare Inequality, even case] \label{PoincareInequalityEven}
There is a universal constant  $P > 0$ independent of $\ell$ so that: Let $f \in W^{1, 2}([-\ell, \ell])$  be an even function satisfying the orthogonality condition:
\begin{align} \label{WeightedOrthogonalityConditionEven}
\int_0^\ell f (s)\cosh^{-2}(s) ds = 0.
\end{align}
Then it holds that 
\begin{align} \notag
e^{+}_\ell(f) & \geq P \int_{0}^\ell f^2(s) \cosh^{-2}(s) ds.
\end{align}
\end{proposition}

\begin{proof}
The proof of the Poincare inequality in the even case is essentially the same as in the odd case:  As before we consider a minimizing sequence $f_j$ for the energy $e^+_{\ell}(-)$. Although in this case, the boundary term  of the energy integral is too large to be absorbed as in the odd case, it appears with the beneficial sign so that, 
\[
e^+(f) = \int_{0}^\ell f'^2 (s) ds - 2 \int_0^\ell f^2(s) \cosh^{-2} (s) ds + f^2 (\ell) \tanh(\ell) \geq - 2 \int_0^\ell f^2(s) \cosh^{-2} (s) ds,
\]
and thus as before $e^+(f)$ is uniformly bounded below. Moreover, since $f_j$ is a minimizing sequence, we have $e^+_\ell(f_j) \leq K$ for a fixed constant $K$ and thus
\[
\int_{0}^\ell f_j'^{2}(s)ds \leq K +  \int_{0}^\ell f_j^2(s)\cosh^{-2}(s)ds -f_j^2 (\ell) \tanh(\ell) \leq K + 2
\]
where again we can discard the boundary term due to the beneficial sign.  Thus, as before the sequence $f_j$ is uniformly bounded in $W^{1, 2}([0, \ell])$ and thus weakly converges to a limit $\phi_\ell$ satisfying the eigenvalue problem.  Arguing as before we take a sequence $\ell_k \rightarrow \infty$ and consider the corresponding sequence of eigenfunctions $\phi_k = \phi_{\ell_k}$. Passing to a subsequence gives smooth convergence  on compact subsets of $[0, \ell]$ to a normalized limit  $\phi$ satisfying
\[
\wt{L} \phi = \lambda \cosh^{-2}(s) \phi, \quad \lambda = \liminf_k \lambda_k.
\]
By Lemma \ref{LEigenvalues}, $\lambda$ must be positive since $\phi$ is even--and thus not $\tanh(s)$--corresponding to $\lambda = 0$--and satisfies the orthogonality condition 
\[
\int_{0}^\ell \phi\cosh^{-2}(s)ds = 0
\]
so that $\phi$ cannot be a constant--corresponding to $\lambda = -2$.
 \end{proof}
 
 \begin{definition}\label{Def:TotalEnergy}
 Given a function $f \in W^{1, 2}([- \ell, \ell])$ we set
 \[
 e_{\ell}(f) = e^{-}_{\ell}(f^-) +   e^{+}_{\ell}(f^+), 
 \]
 where $f^+$ and $f^-$ denote the even and odd parts of $f$, respectively. 
 \end{definition}

 \begin{proposition} \label{PoincareInequalityAgain}
 Let $f \in W^{1, 2}([-\ell, \ell])$ be  a function satisfying the orthogonality conditions
 \[
 \int_{- \ell}^\ell f(s) \tanh(s)\cosh^{-2}(s)ds =  \int_{- \ell}^\ell f(s) \cosh^{-2}(s)ds =0. 
 \]
  Then it holds that
 \begin{align} \notag
 \left( 1 - \frac{2}{2 + P} \right)  \int_{-\ell}^\ell f'^2(s)  \leq 4e_\ell (f).
 \end{align}
 \end{proposition}
 \begin{proof}

The inequality in Proposition \ref{PoincareInequality} can be written explicitly as:
\begin{align} \notag
\int_0^\ell (f^{-})'^2 (s)ds - \int_0^\ell (f^-)^2 (2 + P) \cosh^{-2} (s) - (f^-)^2(\ell)  \tanh'(\ell)/ \tanh(\ell) \geq 0.
\end{align}
Equivalently:
\begin{align} \label{epowe1}
\left(1 - \frac{2}{2 + P} \right)\int_0^\ell (f^{-})'^2 (s) ds + \left( \frac{2}{2 + P} -1 \right)(f^-)^2(\ell)  \tanh'(\ell)/ \tanh(\ell) \leq  e^-(f^-).
\end{align}
Since $(f^-)^2(\ell) \leq \ell \int_0^\ell (f^-)'^2(s) ds$ we have
\begin{align} \notag
\left(1 - \frac{2}{2 + P} \right) \left(1 - \ell  \frac{\tanh'(\ell)}{ \tanh(\ell)}\right)\int_0^\ell (f^{-})'^2 (s) ds  \leq  e^-(f^-).
\end{align}
Similarly,  the inequality of  Proposition \ref{PoincareInequalityEven} can be written
\[
\int_0^\ell (f^{+})'^2 (s)ds - \int_0^\ell (f^+)^2 (2 + P) \cosh^{-2} (s) - (f^+)^2(\ell)  (\cosh^{-1})'(\ell)/ \cosh^{-1}(\ell) \geq 0,
\]
giving
\begin{align}\label{epowe}
\left(1 - \frac{2}{2 + P} \right)\int_0^\ell (f^{+})'^2 (s) ds \leq  e^+(f^+).
\end{align}
Summing (\ref{epowe1}) and (\ref{epowe}) and assuming $\ell$ is sufficiently large so that $\ell \frac{\tanh'(\ell)}{\tanh(\ell)} < 1/2$ gives 
\[
\left(1 - \frac{2}{2 + P} \right)\int_0^\ell (f^{+})'^2 (s) ds +  \left( 1 - \frac{2}{2 + P} \right)  \int_{0}^\ell (f^{-})'^2(s)  \leq 2 e_\ell (f).
\]
 It is easy to verify that 
 \[
\int_0^\ell (f^{+})'^2 (s) ds + \int_0^\ell (f^{-})'^2 (s) ds = \frac{1}{2} \int_{-\ell}^{\ell} f'^2 (s)ds,
 \]
 from which the claim directly follows.
 \end{proof}

 \begin{lemma} \label{fOrthogonalPart}
 Let $f(s)$ be a function in  $L^2([-\ell, \ell])$  satisfying the orthogonality conditions
 \[
 \int_{-\ell}^\ell f(s) \tanh(s) ds =  \int_{-\ell}^\ell f(s) \cosh^{-1}(s) ds = 0.
 \]
 Then we can write
 \begin{align} \label{fDecomp}
 f (s) = \alpha \tanh(s) + \gamma \cosh^{-1}(s)  + g (s)
 \end{align}
 where $g$ satisfies the orthogonality  conditions
 \begin{align} \notag
\int_{- \ell}^\ell g(s) \tanh(s) /\cosh^2(s)  = \int_{- \ell}^\ell g(s)\cosh^{-2}(s) = 0 , 
 \end{align}
 as well as the lower bound
 \[
    \int_0^\ell f^2 (s) ds \leq \int_0^\ell g^2 (s) ds.
 \]
 \end{lemma}
 \begin{proof}
Write
\begin{align} \notag
f = \alpha \tanh(s) + \gamma \cosh^{-1}(s) +   (f - \alpha \tanh(s) - \gamma \cosh^{-1}(s))  := \alpha \tanh(s) + \gamma \cosh^{-1}(s) + g
\end{align}
Then choosing $\alpha$ appropriately, it is clear that we can arrange for $g$ to be orthogonal to $\tanh(s)/\cosh^2(s)$ and $\cosh^{-2}(s)$. In particular, we choose $\alpha$ so that 
\begin{align}\label{AlphaExplicit}
\alpha = \frac{\int_{- \ell}^\ell f(s) \tanh(s)/\cosh^2(s) ds}{\int_{- \ell}^\ell \tanh^2(s)/ \cosh^2(s) ds}, \quad  \gamma = \frac{\int_{- \ell}^\ell f(s) \cosh^{-2}(s) ds}{\int_{- \ell}^\ell  \cosh^{-3}(s) ds}.
\end{align}
Then, the orthogonality conditions on $f$ and Cauchy-Schwartz give
\[
\int_{-\ell}^\ell f^2 ds = \int_{-\ell}^\ell f g \leq \left( \int_{-\ell}^\ell f^2 ds\right)^{1/2}  \left( \int_{-\ell}^\ell g^2 ds\right)^{1/2}.
\]
\end{proof}

\begin{lemma} \label{UnweightedPoincare}
 Let $f$ and $g$ be as in the statement of Lemma \ref{fOrthogonalPart}. Then it holds that 
 \begin{align} \notag
 e_\ell(f)  \geq e_\ell (g) \geq \left( 1 - \frac{2}{2 + P}\right)\left( \int_{- \ell}^{\ell} g'^2(s) ds  \right).
\end{align}
\end{lemma}

\begin{proof}
We have that
\begin{align} \notag
e^{-}_\ell(f^-) = B^- [f^-, f^-]
\end{align}
for the bilinear form 
\[
B^-[f, g] : = \int_{0}^\ell f' (s)g' (s) ds - 2 \int_{0}^\ell f(s) g (s) \cosh^{-2}(s) ds -  f(\ell) g(\ell)  \tanh'(\ell)/\tanh (\ell).
\]
Since $f^{-} = g^- + \alpha \tanh(s)$, we then have
\begin{align} \notag
e^{-}_\ell(f^-)  &  = B^{-}[ \alpha \tanh(s) + g^{-},  \alpha \tanh(s) + g^{-} ]\\ \notag
& = \alpha^2 B^{-}[ \tanh(s),    \tanh(s) ]  + B[ g^{-},   g^{-} ]  + 2 \alpha B[ g^{-},   \tanh(s) ]  \\ \notag
& =  B[g^{-},  g^{-}  ] \\ \notag
& = e^{-}_\ell (g^-).
\end{align}
where the last equality above follows from the definition of $e^{-}_{\ell}(-)$ and $B^-[ -, - ]$ and where the third line follows from the second since $\tanh(s)$ is in the kernel of $\tilde{L}$ with the boundary conditions  (\ref{SolutionRobinBoundaryCondition}) and (\ref{SolutionNeumanBoundaryCondition}). A similar argument applies to the even part of $f$ using the bilinear form
\[
B^+[f, g] : = \int_{0}^\ell f' (s)g' (s) ds - 2 \int_{0}^\ell f(s) g (s) \cosh^{-2}(s) ds +  f(\ell) g(\ell)  \tanh(\ell).
\]
Since $\wt{L} \cosh^{-1}(s) = \cosh^{-1}(s)$, we have
\begin{align*}
e_\ell^+(g^+) & = B^+[g^+, g^+] \\
& =  B^+[f^+ + \gamma \cosh^{-1}(s), f^+  + \gamma \cosh^{-1}(s)] \\
& =  B^+[f^+ , f^+] + 2 \gamma  B^+[f^+,  \cosh^{-1}(s)]  +  \gamma^2B^+[\cosh^{-1}(s), \cosh^{-1}(s)] \\
& = e_{\ell}^+(f^+)  - 2 \gamma \int_0^\ell f^+ \wt{L} \cosh^{-1}(s) ds - \gamma^2  \int_0^\ell \cosh^{-1}(s) \wt{L} \cosh^{-1}(s) ds \\
&=  e_{\ell}^+(f^+)  - 2 \gamma \int_0^\ell f^+ \cosh^{-1}(s) ds - \gamma^2  \int_0^\ell \cosh^{-2}(s) ds\\
& \leq e_{\ell}^+(f^+) 
\end{align*}
where above we have used the orthogonality of $f$ to $\cosh^{-1}(s)$. This completes the proof. 
\end{proof}

We are now ready to prove Proposition \ref{RBCSolutionsIntegralControl}:

\begin{proof}[Proof of Proposition  \ref{RBCSolutionsIntegralControl}]
In the following, in order to conveniently reference  Proposition \ref{PoincareInequality}, Proposition \ref{PoincareInequalityEven}, Proposition \ref{PoincareInequalityAgain} and Lemma \ref{UnweightedPoincare} while using their notational conventions, we set
\begin{align} \notag
f (s,z) = u(s, z) = \alpha(z) \tanh(s)  + \gamma(z) \cosh^{-1}(s)+ g(s, z),
\end{align}
where the multiples $\alpha$ and $\gamma$ are chosen so that 
\begin{align} \notag
\int_{-\ell}^{ \ell} g(s,z) \tanh(s)/\cosh^2(s) ds = \int_{-\ell}^{ \ell} g(s,z)  \cosh^{-2}(s) ds = 0.
\end{align}
Additionally, we will throughout the proof abbreviate $R: = [-\ell, \ell] \times [-N, N]$.  Let $f^{-}$ and $f^{+}$ denote the odd and even parts $f$, respectively, and similarly for $g$, so 
\[
f^{+} = \gamma \cosh^{-1}(s) + g^{+}, \quad  f^{-} = \gamma \cosh^{-1}(s) + g^{-}.
\]
Since the operator $\wt{\mc{L}}$ preserves odd and even functions, we have
\begin{align*}
- \int_{R} f \wt{\mc{L}} f  & = - \int_{R} (f^+ + f^-) \wt{\mc{L}} (f^+ + f^-)\\
& =   - \int_{R} (f^+) \wt{\mc{L}} (f^+)  - \int_{R} (f^-) \wt{\mc{L}} ( f^-) \\
& = -\int_{R} \left\{f^+\left(f^+\right)_{,zz} + f^+ \wt{L} f^+ \right\}  -\int_{R} \left\{f^-\left(f^-\right)_{,zz} + f^- \wt{L} f^- \right\} \\
& = \int_{R}\left(f^+_{z} \right)^2 + 2 \int_{-N}^N e^{+}_{\ell}(f^+) dz  + 2 \int_{R}\left(f^-_{z} \right)^2 + \int_{-N}^N e^{-}_{\ell}(f^-) dz \\
& \geq 2 \int_{-N}^N e^{+}_{\ell}(f^+) dz + 2 \int_{-N}^N e^{-}_{\ell}(f^-) dz \\
& = 2 \int_{-N}^N e_{\ell}(f) dz.
\end{align*}
Using  Lemma \ref{UnweightedPoincare} and the Cauchy Schwartz inequality  we then have:
\begin{align*}
\int_{R} \left(g_{, s}\right)^2 & = \int_{-N}^N \int_{-\ell}^\ell  \left(g_{, s}\right)^2 ds dz \\
& \leq^C  \int_{-N}^N e_{\ell}(g) dz \\
& \leq^C \int_{-N}^N e_{\ell}(f) dz \\ 
& \leq^C - \int_{R} f \wt{\mc{L}}  f \\
& \leq^C \left( \int_{R} f^2 \right)^{1/2}  \left( \int_{R} \left(\wt{\mc{L}}f\right)^2 \right)^{1/2}. \\
\end{align*}
Applying the flat Poincare inequality for the interval $[0, l]$ and  Lemma \ref{fOrthogonalPart}  we have 
\[
\int_{R} f^2 \leq \int_{R} g^2  \leq \ell^2 \int_{R} \left(g_{, s}\right)^2  \leq^C \ell^2\left( \int_{R} f^2 \right)^{1/2}  \left( \int_{R} \left(\wt{\mc{L}}f\right)^2 \right)^{1/2}.
\]
This establishes the $L^2$ estimates asserted in the proposition.  To establish the uniform estimate on the supremum of $u$, pick a point $p_0 = (s_0, z_0)$ where the supremum of $u$ is achieved and set $\ol{u} = |u(p_0)|$. Let $r$ denote the maximal radius such that $u$ restricted to the ball $B_{r} (p_0)$ is at least $\frac{\ol{u}}{2}$. Schauder theory (see Remark \ref{Rem:SchauderTheory}) then gives the bound $\| u\|_{2, \alpha} (p_0) \leq C \ol{u}$  on the localized holder norms, which then gives a uniform lower bound  on $r$. To see this, pick $p \in B_{r} (p_0)$ such that $|u_N(p)| = \ol{u}/2$. Then 
\[
\frac{\ol{u}}{2} = \left|u (p_0) - u (p) \right| \leq \sup_{B_r(p_0)} \left| \nabla u \right||p- p_0| \leq^C r \ol{u}.
\]
We then have 
\[ 
 \ell^4 \left(\left\|\wt{\mc{L}} u \right\|_{L^2}\right)^2  \geq \int_{R} f^2 \geq \int_{B_r(p_0)  \cap R} f^2 \geq c r^2 \ol{u}^2 \geq c \ol{u}^2,
\]
or 
\[
\sup_{R} |u| \leq^C \ell^{2}\left\|\wt{\mc{L}} u \right\|_{L^2}.
\]
This completes the proof. 
\end{proof}

\begin{remark}\label{Rem:SchauderTheory}
Regarding the application of Schauder Theory  in the proof of Proposition \ref{RBCSolutionsIntegralControl}, the relevant results can be found in \cite{GT}. For points $p_0$ away from the boundary, we can apply interior apriori estimates (Theorem 6.2 on page 90). For points $p_0$ near the boundary, we apply the interior estimates on the half ball for the oblique derivative problem (Lemma 6.29 on page 126).
\end{remark}

\section{Solvability of the Poisson problem for strongly orthogonal source terms} \label{Sec:Solvability}
 
  We begin by characterizing the kernel of the stability operator on the rectangles $[-\ell, \ell] \times [-N, N]$ with the imposed boundary conditions. 
 
  \begin{corollary} \label{WhatsInTheRobinKernel}
 Let $\phi: [-\ell, \ell] \times [-N, N] \rightarrow \mb{R}$  be an odd function satisfying:
 \begin{enumerate} 
 \item $\tilde{\mathcal{L}} \phi = 0$. 
\item $\partial_s\phi (\ell, z) \tanh(\ell) - \phi(\ell, z)/\cosh^2(\ell) = 0$.

\item $\partial_z \phi (s, \pm N) = 0 $.
 \end{enumerate}
 Then $\phi$ is a multiple of $\tanh(s)$.
 \end{corollary}
 \begin{proof}
 For $\phi$ as in the hypothesis, we have
 \begin{align*}
\left( \int_{0}^\ell \phi \tanh(s) ds\right)''  & = \int_{0}^\ell \phi_{, z z} \tanh(s) ds \\
 & = -\int_{0}^\ell \wt{L} \phi \tanh(s) ds \\
 & = -\int_{0}^\ell  \phi  \wt{L}\tanh(s) ds - \left. \phi' \tanh(s) \right|_{0}^\ell + \left. \phi \tanh'(s) \right|_{0}^\ell\\
 & = 0.
 \end{align*}
 where the last line follows from the fact that $\wt{L} \tanh(s) = 0$ and the imposed boundary values on $\phi$. Thus the integral $\int_{0}^\ell \phi \tanh(s) ds$ is affine and the imposed boundary conditions at $z = \pm N$ give that is constant. Thus, there is a constant $\alpha$ such that $\wt{\phi} = \phi - \alpha \tanh(s)$ satisfies $\wt{L} \wt{\phi} = 0$ and
 \[
\int_{0}^\ell \phi \tanh(s) ds= 0, \quad \forall z.
 \]
 Applying  Proposition \ref{RBCSolutionsIntegralControl} with $E$  = 0 then gives $\wt{\phi} \equiv 0$ and gives the claim. 
 \end{proof}
 
 \begin{corollary} \label{WhatsInTheRobinKernelAgain}
 Let $\phi: [-\ell, \ell] \times [-N, N] \rightarrow \mb{R}$  be an even function satisfying:
 \begin{enumerate} 
 \item $\tilde{\mathcal{L}} \phi = 0$. 
\item $\partial_s\phi (\ell, z) \cosh^{-1}(\ell) - \phi(\ell, z)\partial_s\left(\cosh^{-1}\right)(\ell) = 0$.
\item $\partial_z \phi (s, \pm N) = 0 $.
 \end{enumerate}
 Then $\phi$ is in the span of $\sin(z) \cosh^{-1}(s)$ and $\cos(z) \cosh^{-1}(s)$.
 \end{corollary}
\begin{proof}
Consider $\phi$ as in the hypothesis and set 
\[
d = d(z) = \int_{0}^{\ell} \phi(s, z) \cosh^{-1} (s) ds.
\]
Then 
\begin{align*}
d'' & =  \left(\int_{0}^{\ell} \phi(s, z) \cosh^{-1} (s) ds \right)'' \\
& = \int_{0}^{\ell} \phi_{, zz}(s, z) \cosh^{-1} (s) ds \\
& = - \int_{0}^\ell \wt{L}\phi \cosh^{-1}(s) ds \\
& = -  \int_{0}^\ell\phi  \wt{L}\cosh^{-1}(s) ds \\
& = -  \int_{0}^\ell\phi \cosh^{-1}(s) ds \\
& = - d.
\end{align*}
Thus after adding a linear combination of $\sin(z) \cosh^{-1}(s)$ and $\cos(z) \cosh^{-1}(z)$ we obtain a function $\wt{\phi}$ satisfying $\wt{L} \wt{\phi} = 0$ and 
\[
 \int_{0}^{\ell} \wt{\phi}(s, z) \cosh^{-1} (s) ds = 0.
\]
Applying  Proposition \ref{RBCSolutionsIntegralControl} with $E = 0$ gives  that  $\sup \wt{\phi}  = 0$, from which it follows that $\phi$ is in the span of $\sin(z) \cosh^{-1}(s)$ and $\cos(z) \cosh^{-1}(z)$.
\end{proof}

  \begin{proposition} \label{StripInvertibilityWithRBC}
Let $E(s, z)$ be a class $C^{0, \alpha}$ function supported on $[-\ell, \ell] \times [- 2\pi, 2 \pi]$ and satisfying
\begin{align} \notag
\int_0^\ell E (s , z) \tanh(s) ds =  \int_0^\ell E (s , z) \cosh^{-1}(s) ds= 0.
\end{align}
Then, given $N > 2 \pi$ there is locally $C^{2, \alpha}$ function $u_N(s, z) : [- \ell, \ell] \times [-N, N] \rightarrow \mb{R}$ satisfying (\ref{SolutionRobinBoundaryCondition}),  (\ref{SolutionRobinBoundaryCondition2}) and  (\ref{SolutionNeumanBoundaryCondition}) and (\ref{SolutionStrongOrthogonality}) in the statement of Proposition \ref{RBCSolutionsIntegralControl} as well as the bounds
\[
\sup \left| u_N\right|, \quad \left\| u_N: L^2\right\| \leq^C \ell^{5/2} \left( \sup\| E\|_{0, \alpha} \right).
\]
 \end{proposition}
 \begin{proof} 
The existence of weak solutions and their higher regularity  follows from standard theory (see \cite{GT}  Theorem 6.30 for regularity in the setting of the oblique derivative problem, and \cite{Ev}, Chapter 6 for existence of weak solutions), and we need only  show that the solutions $u_N$ recover the strong orthogonality conditions (\ref{StrongOrthogonality}). First consider the case that $u_N$ is odd. Then  strong orthogonality  to $\cosh^{-1}(s)$  is trivially satisfied. For the strong orthogonality to $\tanh(s)$, we have:
 \begin{align} \label{SomeStuff}
 \partial^2_{z z}u_N + \partial^2_{s s} u_N + 2 \cosh^{-2} (s) u_N =  \partial^2_{z z}u_N + \tilde{L} u_N = E,
 \end{align}
 where the operator $\wt{L}$ is defined implicitly above. Multiplying both sides by $\tanh(s)$, integrating by parts and using the Robin boundary condition then gives
 \begin{align} \notag
 \int_{-\ell}^\ell  \left(\partial^2_{z z}u_N\right)(s, z) \tanh(s) ds =    \partial^2_{z z}\left(\int_{-\ell}^\ell  u_N(s, z) \tanh(s) ds \right)= 0.
 \end{align}
 The boundary conditions  then immediately imply $\left(\int_0^\ell  u_N(s, z) \tanh (s)ds \right) $ is constant in $z$. The conclusion then follows by adding a multiple of $\tanh(s)$ to $u_N$ if necessary. 
 
 The case that $u_N$ is even is similar: The first orthogonality condition is again trivially satisfied since $\tanh(s)$ is odd. Multiplying  (\ref{SomeStuff}) by $\cosh^{-1}(s)$, using the imposed boundary conditions on the even part gives
  \begin{align} \notag
   \partial^2_{z z}\left(\int_{-\ell}^\ell  u_N(s, z) \cosh^{-1}(s) ds \right)  + \int_{-\ell}^{\ell} u_N \cosh^{-1}(s) ds= 0.
 \end{align}
 Thus, we have
 \[
  \int_{-\ell}^{\ell} u_N \cosh^{-1}(s) ds = a \sin(z) + b \cos(z)
 \]
 for constants $a$ and $b$. Again, modifying $u_N$ by adding multiples of the $\sin(z)\cosh^{-1}(s)$ and $\cos(z)\cosh^{-1}(s)$, which lie in the kernel of $\wt{\mc{L}}$ we have that 
 \[
   \int_{-\ell}^{\ell} u(s, z) \cosh^{-1}(s)  ds = 0, \quad \forall z
 \]
 We can now apply Proposition \ref{RBCSolutionsIntegralControl} to get
 \begin{align*}
 \sup \left|u_N \right|, \quad \left\| u_N\right\|_{L^2} & \leq^C \ell^2 \left\|\wt{\mc{L}} u \right\|_{L^2}  \\
 & \leq \ell^2 \left\| E\right\|_{L^2}  \\
 & \leq \ell^{5/2} \left(\sup\left\|  E\right\|_{0, \alpha}\right)
 \end{align*}
 where the last inequality follows from the fact that  $E$ is supported on $[- \ell,  \ell] \times [- 2\pi , 2\pi]$.
 \end{proof}

\section{Proof of Proposition \ref{RBCSolutionsWeightedControl} } \label{Sec:ProofOfMainTheorem}

 \begin{proof}[Proof of Proposition \ref{RBCSolutionsWeightedControl}]
 Set
 \[
 u = u_{\infty} : = \lim_{N \rightarrow \infty} u_N
 \]
 where $u_N$ is as in the statement of Proposition \ref{StripInvertibilityWithRBC}. Since the supremum of the functions $u_N$ are uniformly locally bounded, they are uniformly locally bounded in $C^{2, \alpha}$ by Schauder theory and thus converge to the $C^{2, \alpha}$ limit $u$ in $C^{2, \alpha/2}$. Thus, $u$ solves the Poisson problem for $\wt{\mc{L}}$ with source term $E$ and satisfies the strong orthogonality conditions (\ref{StrongOrthogonality}). Setting $\left(\sup\| E\|_{0, \alpha} \right) : =  \beta$ we have the  estimates
\[
\sup \left| u\right|, \quad \left\| u: L^2\right\| \leq^C \ell^{5/2} \beta.
\]
For $z > 2\pi$ (away from the support of $E$) that 
\begin{align*}
\frac{1}{2}\left(\int_{0}^\ell u^2 d s \right)_{, zz} &=  \int_{0}^\ell u_{, z}^2 d s + \int_{0}^\ell u_{, zz} u ds \\
& \geq \int_{0}^\ell u \left( -u_{, ss} - 2 \cosh(s) u \right) \\
& = e_{\ell} (u(-, z)) \geq 0.
\end{align*}
where the last line above follows from Lemma \ref{UnweightedPoincare}. Thus the integral $\int_{0}^\ell u^2 ds$ is convex for $|z| > 2 \pi$. The integral bound for the solution $u$   then implies monotonically decreasing in $z$ for $z > 2\pi$. Again, the integral bound Proposition \ref{RBCSolutionsIntegralControl} implies the uniform decay rate of at least
\begin{align}\label{IntMon}
\int_0^{\ell} u^2 ds \leq^C  \beta^2 \ell^{5}\frac{1}{1 + |z|}.
\end{align}
 Observe now that the  $C^0$ bound for the solution $u$ and Schauder Theory (See Remark \ref{Rem:SchauderTheory} again) imply uniform local $C^{2, \alpha}$ norms 
\begin{align}\label{BootStrap1}
\left\| u\right\|_{2, \alpha} (p) \leq^C \ell^{5/2} \beta.
\end{align} 
We can then use the integral bound (\ref{IntMon}) to bootstrap a uniform pointwise decay estimate for  $u_{\infty}$: Let $\ol{u} (z)$ denote the supremum of $\left|u(s,z)\right|$  in $s \in [-\ell, \ell]$ and assume that 
\begin{align} \label{BootStrap3}
\ol{u}(z) \leq^C \ell^{5/2} \beta \left(\frac{1}{1  + |z|}\right)^{p}.
\end{align}
for some $p > 0$. Observe that by (\ref{BootStrap1}), (\ref{BootStrap3}) holds with $p = 0$. Pick $\overline{s}$ and $\underline{s}$ such that $\ol{u}(z) = \left|u_{\infty}(\overline{s}, z)\right|$ and such that  $\ol{u}(z)/2 = \left|u_{\infty}(\underline{s}, z)\right|$. With $d =  |\overline{s} - \underline{s}|$  we have
\begin{align}\label{Bootstrap2}
\frac{\ol{u}}{2} & = \left|\int_{\underline{s}}^{\overline{s}} \left| \frac{\partial u_{\infty}}{\partial s} \right|(s, z) ds \right| \\
& \leq C \ell^{5/2} \beta \left(\frac{1}{1  + |z|} \right)^p  d\notag
\end{align}
We have
\begin{align*}
\beta^3\ell^{15/2}\left(\frac{1}{1  + |z|}\right)^{1 + p} &\geq c \beta \ell^{5/2} \left(\frac{1}{1  + |z|} \right)^p \int_0^\ell u^2 (s, z) ds \\
& \geq \beta \ell^{5/2} \left(\frac{1}{1  + |z|} \right)^p \int_{\underline{s}}^{\overline{s}}  u^2 (s, z) \\
& \geq \beta \ell^{5/2} \left(\frac{1}{1  + |z|} \right)^p \frac{\ol{u}^2}{4} d\\
& \geq c \ol{u}^3
\end{align*}
from which it follows that $\ol{u} \leq^C \ell^{5/2} \beta \left(\frac{1}{1 + |z|}\right)^{\frac{1 + p}{3}}$, and thus (\ref{BootStrap3}) holds with $p$ improved to $\frac{p + 1}{3}$. Iterating, we obtain a sequence of estimates
\[
\ol{u}(z) \leq C(k)\beta \ell^{5/2}\left(\frac{1}{1 + |z|}\right)^{p_k}
\]
where $p_0 = 0$ and $p_{k} = \frac{1 + p_{k - 1}}{3}$. It is easy to check that the sequence $p_{k}$ increases monotonically to the limiting value $p_\infty = 1/2$.
\end{proof}

\bibliographystyle{amsalpha}

\end{document}